\documentclass{patmorin}
\definecolor{brightmaroon}{rgb}{0.76, 0.13, 0.28}
\definecolor{linkblue}{rgb}{0, 0.337, 0.227}
\newcommand{\defin}[1]{\emph{\color{brightmaroon}#1}}
\title{\MakeUppercase{Vertex Ranking of Degenerate Graphs}}
\author{%
  John Iacono%
    \thanks{Université libre de Bruxelles, Belgium.
Supported by the Fonds de la Recherche Scientifique-FNRS.} \and
  Piotr Micek%
    \thanks{Theoretical Computer Science Department, Jagiellonian University, Poland. \email{piotr.micek@uj.edu.pl}.
    Supported by the National Science Center of Poland
under grant UMO-2023/05/Y/ST6/00079 within the OPUS 25 program
      } \and
  Pat Morin%
    \thanks{School of Computer Science, Carleton University, Canada. \email{morin@scs.carleton.ca}. Supported by NSERC.}\and
  Bruce Reed%
    \thanks{Mathematical Institute, Academia Sinica, Taiwan. \email{bruce.al.reed@gmail.com}.  Supported by  NSTC Grant 112-2115-M-001 -013 -MY3}}

\newcommand{\rn}[1]{\chi_{\operatorname{#1-vr}}}

\newcommand{\trn}{\chi_{\mathrm{us}}}
\newcommand{\lrn}{\rn{\ell}}

\newcommand{\texp}{1-1/(\lfloor\ell/2\rfloor+1/2)}
\newcommand{\dexp}{1-\frac{1}{\lfloor\ell/2\rfloor+1/2}}

\begin{document}

\maketitle

\begin{abstract}
  An $\ell$-vertex-ranking of a graph $G$ is a colouring of the vertices of $G$ with integer colours so that in any connected subgraph $H$ of $G$ with diameter at most $\ell$, there is a vertex in $H$ whose colour is larger than that of every other vertex in $H$.  The $\ell$-vertex-ranking number, $\lrn(G)$, of $G$ is the minimum integer $k$ such that $G$ has an $\ell$-vertex-ranking using $k$ colours.  We prove that, for any fixed $d$ and $\ell$, every $d$-degenerate $n$-vertex graph $G$ satisfies $\lrn(G)= O(n^{1-2/(\ell+1)}\log n)$ if $\ell$ is even and $\lrn(G)= O(n^{1-2/\ell}\log n)$ if $\ell$ is odd. The case $\ell=2$ resolves (up to the $\log n$ factor) an open problem posed by \citet{karpas.neiman.ea:on} and the cases $\ell\in\{2,3\}$ are asymptotically optimal (up to the $\log n$ factor).
\end{abstract}

\section{Introduction}

An $\ell$-vertex-ranking of a graph $G$ is a colouring of the vertices of $G$ with integer colours so that in any connected subgraph $H$ of $G$ with diameter at most $\ell$, there is a vertex in $H$ whose colour is larger than that of every other vertex in $H$.  The \defin{$\ell$-vertex-ranking} number $\lrn(G)$ is the minimum integer $k$ such that $G$ has an $\ell$-vertex-ranking $\varphi:V(G)\to\{1,\ldots,k\}$. The \defin{$\ell$-vertex-ranking} number of a graph class $\mathcal{G}$ is $\lrn(\mathcal{G}):=\sup\{\lrn(G):G\in\mathcal{G}\}$.

When $\ell=1$, a colouring is a $1$-vertex-ranking if and only if it is a proper colouring of $G$, so $\chi(G)=\rn{1}(G)$ and, for any $\ell\ge 1$, $\chi(G)\le \lrn(G)$.  Besides the case $\ell=1$, two special cases have received extra attention.  When $\ell=\infty$, $\rn{\infty}(G)$ is equal to the \defin{vertex ranking number}, the \defin{centered chromatic number}, and the \defin{treedepth} of $G$
\cite{nesetril.ossona:tree-depth},
which plays a central role in the theory of sparsity  \cite{nesetril.ossona:sparsity}.
At the other extreme, when $\ell=2$, a $2$-vertex-ranking of $G$ is also known as a \defin{restricted star colouring} \cite{shalu.antony:complexity} or a \defin{unique superior colouring} of $G$ \cite{karpas.neiman.ea:on}.  Previous results for $\ell=2$ and fixed $\ell\ge 2$ are summarized in \cref{previous_works}.

\begin{table}
    \centering{
        \begin{tabular}{|l|c|c|l|} \hline
          \multicolumn{4}{|c|}{$\rn{2}$} \\\hline
            Graph class & Upper Bound & Lower Bound & Ref. \\ \hline
            Trees & $O(\log n/\log\log n)$ & $\Omega(\log n/\log\log n)$ & \cite{karpas.neiman.ea:on} \\
            Planar graphs & $O(\log n/\log^{(3)} n)$ & $O(\log n/\log^{(3)} n)$ & \cite{bose.dujmovic.ea:asymptotically} \\
            Proper minor closed & $O(\log n)$ & $\Omega(\log n/\log\log n)$ & \cite{karpas.neiman.ea:on} \\
            $d$-cubes & $d+1$ & $d+1$ & \cite{almeter.demircan.ea:graph} \\
            Max-degree 3 & $7$ & & \cite{almeter.demircan.ea:graph} \\
            Max-degree $\Delta$ & $O(\min\{\Delta^2,\Delta\sqrt{n}\})$ & $\Omega(\Delta^2/\log \Delta)$ & \cite{karpas.neiman.ea:on,almeter.demircan.ea:graph} \\
            $d$-degenerate & $O(d\sqrt{n})$ & $\Omega(n^{1/3} + d^2/\log d)$ & \cite{karpas.neiman.ea:on,almeter.demircan.ea:graph} \\
            \hline \multicolumn{4}{c}{} \\
            \hline
            \multicolumn{4}{|c|}{$\rn{\ell}$, for fixed $\ell\ge 2$} \\\hline
            Simple treewidth $\le t$ & $O(\log n/\log^{(t)} n)$ & $\Omega(\log n/\log^{(t)} n)$ & \cite{bose.dujmovic.ea:asymptotically} \\
            Treewidth $\le t$ & $O(\log n/\log^{(t+1)} n)$ & $\Omega(\log n/\log^{(t+1)} n)$ & \cite{bose.dujmovic.ea:asymptotically} \\
            Planar graphs & $O(\log n/\log^{(3)} n)$ & $\Omega(\log n/\log^{(3)} n)$ & \cite{bose.dujmovic.ea:asymptotically} \\
            Outerplanar graphs & $O(\log n/\log^{(2)} n)$ & $\Omega(\log n/\log^{(2)} n)$ &  \cite{bose.dujmovic.ea:asymptotically,karpas.neiman.ea:on} \\
            Genus-$g$ graphs & $O(g\log n/\log^{(3)} n)$ & $\Omega(\log n/\log^{(3)} n)$ & \cite{bose.dujmovic.ea:asymptotically} \\
            \hline
        \end{tabular}
    } %
    \caption{Summary of previous results on $\rn{2}$ and $\lrn$ for fixed $\ell\ge 2$. Here, $\log^{(c)} n:=\underbrace{\log\log\cdots\log}_c n$.
    }
\label{previous_works}
\end{table}

The current work is motivated by the gap between the upper and lower bounds for $\rn{2}(G)$ for $d$-degenerate graphs $G$. For fixed $d\ge 2$, the upper bound is $O(\sqrt{n})$, while the lower bound is $\Omega(n^{1/3})$. Closing this gap is stated explicitly as an open problem by \citet{karpas.neiman.ea:on} and \citet{bose.dujmovic.ea:asymptotically}. Our first contribution solves
this problem, up to a logarithmic factor.

\begin{thm}\label{unique_superior}
  For any positive integer $d$ there exists a constant $c:=c(d)$ such that, for every integer $n\ge d$, every $n$-vertex $d$-degenerate graph $G$ satisfies $\rn{2}(G)\le cn^{1/3}\log n$.
\end{thm}

\Cref{unique_superior} is an immediate consequence of the following more general upper bound for the $\ell$-vertex-ranking number of $d$-degenerate graphs.

\begin{thm}\label{l_d_degenerate_upper_bound}
  For any positive integers $d$ and $\ell$ there exists a constant $c:=c(d,\ell)$ such that, for every integer $n\ge d$, every $n$-vertex $d$-degenerate graph $G$ satisfies
  \[
    \lrn(G)\le c n^{\dexp}\log n
    = \begin{cases}
      cn^{1-\frac{2}{\ell}}\log n & \text{if $\ell$ is odd} \\
      cn^{1-\frac{2}{\ell+1}}\log n & \text{if $\ell$ is even.}
      \end{cases}
  \]
\end{thm}
\Cref{l_d_degenerate_upper_bound} exhibits a parity phenomenon one encounters when counting the number of paths of length $\ell$ in a $d$-degenerate graph of maximum-degree $\Delta$. Because of this, the bound in \cref{l_d_degenerate_upper_bound} is the same for $\ell=2k$ and $\ell=2k+1$, for any positive integer $k$.  One might think that this is just an artifact of the proof technique and that the bound for even values of $\ell$ is not tight.  However, \citet{karpas.neiman.ea:on} proved the existence of $2$-degenerate $n$-vertex graphs with $\rn{2}(G)\in\Omega(n^{1/3})$ and \cref{l_d_degenerate_upper_bound} (with $\ell=2$) matches this lower bound to within a logarithmic factor.  Since $\rn{3}(G)\ge\rn{2}(G)$ for any graph $G$, it also matches this bound when $\ell=3$.  Thus, \cref{l_d_degenerate_upper_bound} is tight (up to a $\log n$ factor) for $\ell=2$ and $\ell=3$, leading us to suspect that it is tight for any fixed $\ell$.

Like the upper bound on $\rn{2}(G)$ in \cite{karpas.neiman.ea:on}, \cref{l_d_degenerate_upper_bound} follows quickly from a theorem about graphs that are both $d$-degenerate and have maximum-degree $\Delta$. For such graphs, we prove:

\begin{thm}\label{l_degenerate_and_degree}
  For all positive integers $d$ and $\ell\ge 2$ there exists a constant $c:=c(d,\ell)$ such that, for every integer $\Delta\ge d$ and every integer $n\ge \Delta$, every $n$-vertex $d$-degenerate graph $G$ of maximum-degree at most $\Delta$ satisfies
  \[
    \lrn(G)\le c\Delta^{\lfloor\ell/2\rfloor-1/2}\log^{5/4} n
    =\begin{cases}
    c\Delta^{\ell/2-1}\log^{5/4} n & \text{if $\ell$ is odd.} \\
    c\Delta^{\ell/2-1/2}\log^{5/4} n & \text{if $\ell$ is even.}
    \end{cases}
  \]
  Furthermore, if $\Delta^{\lfloor\ell/2\rfloor-1}\ge\log n$ then $\lrn(G)\le c\Delta^{\lfloor\ell/2\rfloor-1/2}\log n$.
\end{thm}

\Cref{l_d_degenerate_upper_bound} follows from \cref{l_degenerate_and_degree}, by the following easy argument:  Let $S$ be the set of vertices in $G$ having degree at least $\Delta$, for some carefully chosen value of $\Delta$.  Since $G$ is $d$-degenerate, it has at most $dn$ edges, so the total degree of all vertices in $G$ is at most $2dn$, so $|S|\le 2dn/\Delta$.  Now apply \cref{l_degenerate_and_degree} to the graph $G-S$ which is $d$-degenerate and has maximum-degree $\Delta$ to obtain a colouring $\varphi:V(G-S)\to\{1,\ldots,k\}$.  Finally, colour every vertex in $S$ with a distinct colour larger than $k$.  In this way, the total number of colours used is $k + |S|\le c\Delta^{\lfloor\ell/2\rfloor-1/2}\log^{5/4} n + 2dn/\Delta$. Choosing $\Delta$ to balance these two quantities yields \cref{l_d_degenerate_upper_bound}.  This argument is presented in a little more detail at the end of \cref{the_proof}.

\section{Preliminaries}

For any standard graph-theoretic terminology and notation not defined here, we use the same conventions used in the textbook by \citet{diestel:graph}.  A graph $G$ has vertex set $V(G)$ and edge set $E(G)$.  For any $S\subseteq V(G)$, $G[S]$ denote the subgraph of $G$ induced by the vertices in $S$.  For any vertex $v$ of $G$, $N_G(v):=\{w:vw\in E(G)\}$ and $\deg_G(v):=|N_G(v)|$.  For an integer $\ell$, $G^\ell$ denotes the graph with vertex set $V(G)$ that contains an edge $vw$ if and only if some path in $G$ with at most $\ell$ edges contains both $v$ and $w$.

The following alternative definition of $\ell$-vertex-ranking turns out to be more convenient for proofs and is what we will use from this point on.
\begin{obs}\label{alternate_def}
  A vertex colouring $\varphi:V(G)\to\{1,\ldots,k\}$ of a graph $G$ is an $\ell$-vertex ranking of $G$ if and only if, for each path $v_0,\ldots,v_r$ in $G$ with at most $\ell$ edges, $\varphi(v_0)\neq \varphi(v_r)$ or $\max\{\varphi(v_1),\ldots,\varphi(v_{r-1})\}>\varphi(v_0)$.
\end{obs}

For a directed graph $G$, we write $\overrightarrow{vw}$ to denote the directed edge with source $v$ and target $w$.  For a vertex $v$ in a directed graph $G$, $N^+_{G}(v):=\{w\in V(G):\overrightarrow{vw}\in E(G)\}$ denotes the set of out-neighbours of $v$ and $N^-_G(v):=\{u\in V(G):\overrightarrow{uv}\in E(G)\}$ denotes the set of in-neighbours of $v$, $\deg^+_{G}(v):=|N^+_G(v)|$ is the out-degree of $v$, and $\deg^-_{G}(v):=|N^-_G(v)|$ is the in-degree of $v$. We also define $N^+_{G}[v]:=\{v\}\cup N^+_{G}(v)$ and $N^-_{G}[v]:=\{v\}\cup N^-_{G}(v)$ to be the closed out- and in-neighbourhoods of $v$, respectively.

We repeatedly make use of the following foklore result:

\begin{obs}\label{orientation_to_degeneracy}
  If an undirected graph $G$ has an orientation in which each vertex has out-degree at most $d$, then $G$ is $2d$-degenerate.
\end{obs}

\begin{proof}
  For any $S\subseteq V(G)$, the orientation shows that the induced subgraph $G[S]$ contains at most $d|S|$ edges and therefore the total degree of all vertices in $G[S]$ is at most $2d|S|$.  Therefore, for any $S\subseteq V(G)$, the induced graph $G[S]$ has a vertex of degree at most $2d$.
\end{proof}

An \defin{undirected path} $\Pi$ in a graph $G$ is a tree with exactly two leaves whose edges are all edges of $G$. The leaves of $\Pi$ are called the \defin{endpoints} of $\Pi$. The \defin{length} of $\Pi$ is the number of edges in $\Pi$, which is exactly one less than the number of vertices.  With a slight abuse of notation, we write $x_0,\ldots,x_r$ to denote a length-$r$ undirected path $\Pi$ where $x_0$ and $x_r$ are the endpoints of $\Pi$ and $\Pi$ contains the edge $x_{i-1}x_i$ for each $i\in\{1,\ldots,r\}$.  Note that each undirected path $\Pi:=x_0,\ldots,x_r$ in $G$ corresponds to exactly two paths $x_0,x_1,\ldots,x_r$ and $x_r,x_{r-1},\ldots,x_0$ in $G$.

For a graph $G$, let $\mathcal{P}_\ell(G)$ denote the set of all undirected paths of length at most $\ell$ in $G$.  The set $\mathcal{P}_\ell(G)$ is critical for us since, by \cref{alternate_def}, this is precisely the set of paths that need to be considered to determine if a vertex-colouring $\varphi$ of $G$ is an $\ell$-vertex-ranking.  The following lemma shows that the paths in $\mathcal{P}_\ell(G)$ can be mapped onto their endpoints in such a way that no endpoint receives too many paths.  Its proof uses a technique introduced by Cairns to upper bound the number of length-$\ell$ paths in planar graphs (see also \cite[Lemma~5]{devroye.dujmovic.ea:notes}).

\begin{lem}\label{advanced_cairns}
  For any integers $d\ge 2$, $\ell\ge 2$, $\Delta\ge d$ and any
  graph $G$ of maximum-degree $\Delta$ that has an orientation of maximum out-degree $d$ there exists a mapping $\rho:\mathcal{P}_\ell(G)\to V(G)$ such that
  \begin{compactenum}[(i)]
    \item $\rho(\Pi)$ is an endpoint of $\Pi$ for each $\Pi\in\mathcal{P}_\ell(G)$; and
    \item $|\rho^{-1}(v)| \le 2^{\ell+1}d^{\lceil \ell/2\rceil}\Delta^{\lfloor\ell/2\rfloor}$ for each $v\in V(G)$.
  \end{compactenum}
\end{lem}

\begin{proof}
  Fix some orientation of $G$ of maximum out-degree $d$. For each $\Pi\in\mathcal{P}_\ell(G)$ let $x_0,\ldots,x_r$ be one of the two paths in $G$ that corresponds to $\Pi$, chosen so that the edge $x_{i-1}x_{i}$ is oriented from away from $x_{i-1}$ and towards $x_{i}$ for at least half the indices $i\in\{1,\ldots,r\}$.  When $x_{i-1}x_{i}$ is oriented away from $x_{i-1}$, we call it a \defin{downstream edge} of $\Pi$. Otherwise we call $x_{i-1}x_{i}$ an \defin{upstream edge} of $\Pi$. In other words, we choose the endpoint $x_0$ so that at least half the edges of $\Pi$ are downstream edges, and we set $\rho(\Pi):=x_0$.  Observe that the path $x_0,\ldots,x_r$ can be uniquely reconstructed from the following information:
  \begin{compactenum}[(a)]
    \item A sequence $b_1,\ldots,b_r$ of $r$ bits, where $b_i=1$ if $x_{i-1}x_i$ is a downstream edge of $\Pi$ and $b_i=0$ if $x_{i-1}x_i$ is an upstream edge of $\Pi$.\label{bitstring}
    \item A sequence $\delta_1,\ldots,\delta_r$ of integers, where $\delta_i\in\{1,\ldots,d\}$ if $b_i=1$ and $\delta_i\in\{1,\ldots,\Delta\}$ if $b_i=0$.  The integer $\delta_i$ uniquely identifies the neighbour $x_i$ of $x_{i-1}$ so that, starting at $x_0$ we can uniquely reconstruct the path $x_0,\ldots,x_r$ which uniquely identifies the undirected path $\Pi$.\label{directions}
  \end{compactenum}
  The number of choices for (\ref{bitstring}) is $2^r$.  Since (\ref{bitstring}) has at least as many $1$-bits as $0$-bits, the number of choices for (\ref{directions}) is at most $d^{\lceil r/2\rceil}\Delta^{\lfloor r/2\rfloor}$.  So the number of paths of length $r$ in $\rho^{-1}(x_0)$ is at most $2^rd^{\lceil r/2\rceil}\Delta^{\lfloor r/2\rfloor}$.  Summing $r$ over $1$ to $\ell$ completes the proof.
\end{proof}

Observe that, for each edge $vw$ in $G^{\ell}$ there is at least one path in $\mathcal{P}_\ell(G)$ with endpoints $v$ and $w$.  For each edge $vw$ of $G^{\ell}$ we can select one such representative path $\Pi_{vw}\in\mathcal{P}_\ell(G)$.  If we then orient each edge $vw$ of $G^{\ell}$ away from $\rho(\Pi_{vw})$ then we get an orientation of $G^{\ell}$ in which each vertex has out-degree at most $2^{\ell+1}d^{\lceil \ell/2\rceil}\Delta^{\lfloor\ell/2\rfloor}$.  Combined with \cref{orientation_to_degeneracy} this gives the following corollary:

\begin{cor}\label{degeneracy_of_g_l}
  For any integers $d\ge 1$, $\Delta\ge d$ and any $d$-degenerate graph $G$ of maximum degree $\Delta$, $G^{\ell}$ is $2^{\ell+2}d^{\lceil \ell/2\rceil}\Delta^{\lfloor\ell/2\rfloor}$-degenerate.
\end{cor}

In order to obtain the bound in \cref{l_degenerate_and_degree} we need a special version of \cref{advanced_cairns} that only considers undirected paths $v_0,\ldots,v_r\in \mathcal{P}_\ell(G)$ such that $v_0v_1$ is directed toward $v_0$ and $v_{r-1}v_r$ is directed toward $v_r$.  For a directed graph $G'$ whose underlying undirected graph is $G$, let $\widehat{\mathcal{P}}_\ell(G')$ denote the set of undirected paths $x_0,\ldots,x_r$ in $G$ of length $r\le\ell$ and such that $\overleftarrow{x_0x_1}$ is an edge of $G'$ and $\overrightarrow{x_{r-1}x_{r}}$ is also an edge of $G'$.

\begin{lem}\label{advanced_cairns2}
  For any integers $d\ge 2$, $\ell\ge 3$, $\Delta\ge d$ and any
  directed graph $G'$ of maximum degree $\Delta$ and maximum out-degree $d$, there exists a mapping $\gamma:\widehat{\mathcal{P}}_\ell(G')\to V(G')$ such that
  \begin{compactenum}[(i)]
    \item $\gamma(\Pi)\in V(\Pi)$ for each $\Pi\in\mathcal{P}$; and
    \item $|\gamma^{-1}(v)| \le 2^{\ell}d^{\lceil \ell/2\rceil+1}\Delta^{\lfloor\ell/2\rfloor-1}$ for each $v\in V(G')$.
  \end{compactenum}
\end{lem}

\begin{proof}
  The proof is almost identical to \cref{advanced_cairns2} with two modifications.  If $G'$ contains both edges $\overrightarrow{vw}$ and $\overrightarrow{wv}$ then this edge is considered as a downstream edge no matter which direction it is traversed, since there are at most $d$ options for edges leaving $v$ (one of which is $w$) and at most $d$ options for edges leaving $w$ (one of which is $v$).

  For an undirected path $\Pi:=v_0,\ldots,v_r\in\widehat{\mathcal{P}}_\ell(G')$, the total number of upstream edges in the path $v_1,\ldots,v_r$ and in the path $v_{r-1},\ldots,v_0$ is at most $r-2$. (The edge $v_{r-1}v_r$ is a downstream edge in the first path, $v_1v_0$ is a downstream edge in the second path, and each of the edges in $v_1,\ldots,v_r$ is upstream in at most one of the two paths.)  Therefore we can choose the endpoint $v_0$ so that the number of upstream edges in $v_1,\ldots,v_r$ is at most $\lfloor (r-2)/2\rfloor=\lfloor r/2\rfloor-1$ and set $\gamma(\Pi):=v_1$. For $r\ge 3$, $\Pi$ can then be obtained by taking the union of the paths $x_1,x_0$ and $x_1,x_2,\ldots,x_r$.  The first path consists of one downstream edge, so there are at most $d$ options for the first path.  The second path has $r-1$ edges and at most $\lfloor r/2\rfloor-1$ of these are upstream edges, so there are at most $2^{r-1}d^{\lceil r/2\rceil}\Delta^{\lfloor r/2\rfloor-1}$ options for the second path.
   Thus, the number of paths of length $r$ assigned to any vertex is at most $2^{r-1}d^{\lceil r/2\rceil+1}\Delta^{\lfloor r/2\rfloor-1}$ for $r\ge 3$ (and at most $d^r$ for $r\in\{1,2\}$). Again, the proof finish by summing over $r$ in $1,\ldots.\ell$.
\end{proof}

\section{The Proof}
\label{the_proof}

For a vertex colouring $\varphi:V(G)\to\N$ of a graph $G$, we say that an undirected path $\Pi:=x_0,\ldots,x_r$ in $G$ is an \defin{$\ell$-violation} of $\varphi$ if $\Pi$ has length $r\le\ell$ and $\varphi(x_0)=\varphi(x_r)=\max\{\varphi(x_0),\ldots,\varphi(x_r)\}$.  Observe that $\varphi$ is an $\ell$-vertex-ranking if and only if $G$ contains no $\ell$-violations of $\varphi$.

\begin{proof}[Proof of \cref{l_degenerate_and_degree}]
  Let $G$ be an $n$-vertex $d$-degenerate graph of maximum-degree $\Delta$.  Let $S_0:=V(G)$ and, for each integer $i\ge 1$, let $S_i:=\{v\in S_{i-1}:\deg_{G[S_{i-1}]}(v)\ge 4d\}$.  Since $G$ is $d$-degenerate $G[S_{i-1}]$ has at most $d|S_{i-1}|$ edges.  Therefore
  \[
    2d|S_{i-1}|\ge \sum_{v\in S_{i-1}} \deg_{G[S_{i-1}]}(v)\ge 4d|S_i| \enspace ,
  \]
  so $|S_i|\le |S_{i-1}|/2\le n/2^i$ for each $i\ge 1$.  Let $q$ be the maximum integer such that $S_q$ is non-empty.  Since $1\le |S_q|\le n/2^q$, $q\le \log_2 n$.  For each $i\in\{0,\ldots,q\}$, let $L_i:=S_i\setminus S_{i+1}$.  (These notations are mnemonics: $S_i$ contains the \defin{survivors} of round $i-1$ and $L_i$ is \defin{layer} $i$.)

  Let $b=1/4$ and let $k:=\Delta^{\lfloor\ell/2\rfloor-1/2}\log^b n$, so that our goal is to find an $\ell$-vertex-ranking of $G$ using $O(k\log n)$ colours. We compute our colouring using a two phase algorithm. In the first phase we use a sequence of pairwise-disjoint colour palettes $\Phi_0,\ldots,\Phi_{q}$, each of size $2k$, such that for each $1\le i < j\le q$, every colour in $\Phi_i$ is less than every colour in $\Phi_j$.  We will use the colours in $\Phi_i$ to colour the vertices in $L_i$, for each $i\in\{0,\ldots,q\}$.  The total number of colours used in this phase is $2k(q+1)\le 2k(1+\log n)= O(k\log n)$.  The first phase colouring $\varphi:V(G)\to\Phi_0\cup\cdots\cup\Phi_q$ may have some $\ell$-violations that will be eliminated by re-colouring some vertices in the second phase using an additional palette $\Phi_{q+1}$ of size $O(k\log n)$.

  Let $\mathcal{P}$ contain all the undirected paths $x_0,\ldots,x_r$ in $\mathcal{P}_{\ell}(G)$ such that $x_0$ and $x_r$ are in the same layer $L_j$ and $\{x_1,\ldots,x_{r-1}\}\subseteq \bigcup_{i=0}^j L_i$.  Since $G$ is $d$-degenerate, it has an orientation of maximum out-degree $d$.  Let $\rho:\mathcal{P}\to V(G)$ be the mapping given by \cref{advanced_cairns}, applied to $G$, restricted to the paths in $\mathcal{P}$. (The purpose of the restriction is so that $\rho^{-1}(v)$ denotes a subset of $\mathcal{P}$.)  We will say that an undirected path $\Pi:=x_0,\ldots,x_r$ in $\mathcal{P}$ is \defin{problematic} if $x_0$ and $x_r$ are assigned the same colour in the first phase of the algorithm.

  Although we have not yet completed the description of the first phase colouring procedure, we already know enough to establish the following claim:  Any $\ell$-violation $\Pi$ of the first phase colouring $\varphi$ is a problematic path in $\mathcal{P}$.  Indeed, if $\Pi=x_0,\ldots,x_r$ is an $\ell$-violation, then $\varphi(x_0)=\varphi(x_r)\in \Phi_j$ for some $j\in\{0,\ldots,q\}$.  Thus $x_0,x_r\in L_j$ for some $j\in\{0,\ldots,q\}$.  Furthermore, since $\Pi$ is an $\ell$-violation $\varphi(x_0)=\max\{\varphi(x_0),\ldots,\varphi(x_r)\}$, so no colour in $\Phi_{j+1},\ldots,\Phi_q$ appears at any vertex in $x_1,\ldots,x_{r-1}$.  Therefore, $\{x_1,\ldots,x_{r-1}\}\subseteq \bigcup_{i=0}^j L_i$.  Therefore $\Pi\in\mathcal{P}$.  Since $\varphi(x_0)=\varphi(x_r)$, $\Pi$ is problematic.

  We now describe the first phase colouring algorithm.  Consider the multigraph $G^*$ that, for each $vw\in V(G)$ contains as many copies of the edge $vw$ as there are undirected paths in $\mathcal{P}$ with endpoints $v$ and $w$.  The existence of $\rho$ implies that $G^*$ has an orientation in which each vertex has out-degree $O(\Delta^{\lfloor\ell /2\rfloor})$.  By \cref{orientation_to_degeneracy}, $G^*$ is $O(\Delta^{\lfloor\ell /2\rfloor})$-degenerate.  Label the vertices of $G$ as $v_1,\ldots,v_n$ so that $v_i$ has degree $O(\Delta^{\lfloor\ell /2\rfloor})$ in $G^*[\{v_1,\ldots,v_{i}\}]$. In other words, for each $a\in\{1,\ldots,n\}$ the number of paths in $\mathcal{P}$ with one endpoint at $v_a$ and the other endpoint in $\{v_1,\ldots,v_{a-1}\}$ is $O(\Delta^{\lfloor\ell /2\rfloor})$.  In the first phase, we will colour the vertices one by one in the order $v_1,\ldots,v_n$.

  Let $\tau:\mathcal{P}\to V(G)$ be the mapping that maps $\Pi\in\mathcal{P}$ to $v_b$ if and only if the endpoints of $\Pi$ are $v_a$ and $v_b$ and $a < b$.\footnote{The two mappings $\tau$ and $\rho$ are similar. The difference is that $\rho$ corresponds to some orientation of $G^*$ and $\tau$ corresponds to an acyclic orientation of $G^*$.}  Consider some vertex $w\in L_j$.  For each $\alpha \in \Phi_j$, let $N_{\alpha}(w)$ be the number of paths in $\tau^{-1}(w)$ whose other endpoint (not $w$) is assigned the colour $\alpha$. If $w=v_b$ then $N_{\alpha}(w)$ is completely determined by the colours of $v_1,\ldots,v_{b-1}$, so the value of $N_\alpha(v_b)$ is determined after $v_{b-1}$ is coloured but before $v_b$ is coloured.  Observe that assigning the colour $\alpha$ to $w$ creates exactly $N_{\alpha}(w)$ problematic paths, and these are all in $\tau^{-1}(w)$.  Therefore, $\sum_{\alpha\in\Phi_j} N_\alpha(w)=|\tau^{-1}(w)|\in O(\Delta^{\lfloor\ell /2\rfloor})$.

  For each vertex $w$, we choose the colour of $w$ uniformly at random from a subpalette of $\Phi_i$ that contains exactly half of the $2k$ colours in $\Phi_i$.  Specifically, we choose the colour of $w$ from a palette $\Phi(w)\subset \Phi_i$ that contains $k$ colours $\alpha$ in $\Phi_i$ with the smallest $N_\alpha(w)$ values, so that  $\max\{N_\alpha(w):\alpha\in \Phi(w)\}\le\min\{N_\alpha(w):\alpha\in\Phi_i\setminus\Phi(w)\}$.  Let $M:=c\Delta^{\lfloor\ell/2\rfloor}/k$ with $c$ sufficiently large so that $Mk\ge \tau^{-1}(w)$.
  Then
  \[
    Mk \ge |\tau^{-1}(w)|= \sum_{\alpha\in \Phi_i} N_\alpha(w) \ge \sum_{\alpha\in\Phi_i\setminus\Phi(w)}N_\alpha(w) \ge k\max\{N_\alpha(w):\alpha\in \Phi(w)\} \enspace .
  \]
  Therefore, $\max\{N_\alpha(w):\alpha\in \Phi(w)\}\le M$ is the maximum number of problematic paths in $\mathcal{P}$ that can be created by colouring $w$, and $M$ is the maximum number of problematic paths that can be created by colouring any single vertex.  This completes the description of the first-phase colouring $\varphi$ of $G$.

  The first phase colouring $\varphi$ is not yet an $\ell$-vertex-ranking.  Our goal now is to study the problematic paths in $\mathcal{P}$, each of which is a potential $\ell$-violation of $\varphi$.  For each problematic $\Pi\in\mathcal{P}$, we will choose a vertex $y$ of $\Pi$ to recolour in the second round to eliminate the possibility that $\Pi$ is an $\ell$-violation. of $\varphi$.  Consider the directed graph $G'$ obtained from $G$ by adding each edge $\overrightarrow{vw}$ if $vw\in E(G)$, $v\in L_i$, $w\in L_j$ and $i\le j$.  (Note that if $v$ and $w$ are in the same layer $L_j$ then both edges $\overrightarrow{vw}$ and $\overleftarrow{vw}$ are present in $G'$.)  Then $G'$ has maximum out-degree at most $4d$ and $\widehat{\mathcal{P}}_\ell(G')$ contains $\mathcal{P}$. Let $\gamma:\mathcal{P}\to V(G)$ be the result of applying \cref{advanced_cairns2} to $G'$ (and then restricting it to the paths in $\mathcal{P}$.)  For each problematic path $\Pi\in \mathcal{P}$, we call the vertex $\gamma(\Pi)$ \defin{problematic}.  We will recolour $\gamma(\Pi)$ in order to avoid the potential $\ell$-violation at $\Pi$.

  Let $P$ be the set of all problematic vertices.  In the second phase, we recolour every vertex in $P$ with a colour in a palette $\Phi_{q+1}$ of size $ck\log n + 1$ whose colours are all larger than all colours in $\Phi_0,\ldots,\Phi_q$.  Since each $\ell$-violation of $\varphi$ contains a vertex in $P$, this recolouring eliminates all the existing $\ell$-violations in $\varphi$.  More precisely, after this recolouring any remaining $\ell$-violation must have both endpoints whose colour is in $\Phi_{q+1}$.  So that this never happens, we will ensure that our recolouring is a proper colouring of $G^\ell[P]$.  By definition, this means that any path in $\mathcal{P}_\ell(G)$ with both endpoints in $P$ has endpoints of different colours and is therefore not an $\ell$-violation.

  By \cref{degeneracy_of_g_l}, $G^\ell$ is $O(\Delta^{\lfloor\ell /2\rfloor})$-degenerate.  Let $H$ be a directed acyclic graph obtained from $G^{\ell}$ in which each vertex has out-degree $O(\Delta^{\lfloor\ell /2\rfloor})$. We want to show that $G^\ell[P]$ has chromatic number at most $ck\log n+1$.
  To do this, we will show that the maximum out-degree of $H[P]$ is at most $ck\log n$, with high probability.  In fact, we will show something stronger: that every vertex $p$ in $H$ has at most $ck\log n$ out-neighbours in $P$.

  From this point on, we fix some vertex $p$ of $H$ and study the random variable $|N_H^+(p)\cap P|$.  Instead of focusing on the set $P$ of problematic vertices, we focus instead on problematic paths.  For each problematic vertex $y$, some path $\Pi$ in $\gamma^{-1}(y)$ is problematic. Therefore,
  \begin{equation}
    |N^+_H(p)\cap P| \le \sum_{y\in N^+_H(p)} \left|\left\{\Pi\in\gamma^{-1}(y): \text{$\Pi$ is problematic}\right\}\right| =: X'_{p} \enspace .
    \label{path_upper_bound}
  \end{equation}
  Each path $\Pi\in\mathcal{P}$ is problematic with probability at most $1/k$ since $\Pi$ becomes problematic precisely when we choose the same colour for $w:=\tau(\Pi)$ that was already chosen for the other endpoint of $\Pi$.  Therefore,
  \begin{align*}
    \E\left(|N^+_H(p)\cap P|\right)
    & \le \frac{1}{k}\sum_{y\in N^+_H(p)}|\gamma^{-1}(y)| \\
    & \le \frac{1}{k}\cdot|N^+_H(p)|\cdot O(\Delta^{\lfloor\ell /2\rfloor-1}) \\
    & \le \frac{1}{k}\cdot O(\Delta^{2\lfloor\ell /2\rfloor-1}) \\
    & = O(\Delta^{\lfloor\ell /2\rfloor-1/2}\log^{-b} n) = O(k\log^{-2b} n) \enspace .
  \end{align*}
  This is a good sign. If the event set $\{\text{``$\Pi$ is problematic''}:\Pi\in\mathcal{P}\}$ were mutually independent then it would be a simple matter of applying a Chernoff bound.  Unfortunately this is not the case since, for each vertex $w$, all of the events in $\{\text{``$\Pi$ is problematic''}:\Pi\in\tau^{-1}(w)\}$ are all affected by the choice of colour for $w$.

  The remainder of the proof is more probability than graph theory.  We will use a tail estimate for sums of independent random variables due to Bernstein that is applicable when these random variables have sufficiently small variance.  The statement of Bernstein's Inequality and the calculations needed to apply it in this context are deferred to the next section. We use the rest of our time here to describe a random variable $X_p$ that stochastically dominates $X'_p\ge |N^+_H(p)\cap P|$ and is a sum of independent random variables.\footnote{A random variable $X$ \defin{stochastically dominates} a random variable $Y$ if $\Pr(X\ge x) \ge \Pr(Y\ge x)$ for all $x\in\R$.}

  For each vertex $w$ of $G$, let $\mathcal{P}_{p,w}:=\tau^{-1}(w)\cap (\bigcup_{y\in N^+_H(p)}\gamma^{-1}(y))$ and define
  \[
    X'_{p,w} :=\left|\left\{\Pi\in\mathcal{P}_{p,w}:\text{$\Pi$ is problematic}\right\}\right|
    \enspace .
  \]
  Since each problematic path $\Pi$ counted by the right-hand-side of \cref{path_upper_bound} becomes problematic when $w:=\tau(\Pi)$ is coloured, $X'_p=\sum_{w\in V(G)} X'_{p,w}$.

  Suppose $w\in L_j$ for some $j\in\{0,\ldots,q\}$.  For each $\alpha\in\Phi_j$, let $N_\alpha(p,w)$ be the number of paths in $\mathcal{P}_{p,w}$ that would have become problematic if we had set the colour of $w$ to $\alpha$. Then $\sum_{\alpha\in\Phi(w)} N_\alpha(p,w) \le \sum_{\alpha\in\Phi_j} N_\alpha(p,w) \le \sum_{\alpha\in\Phi_j} N_\alpha(w) = |\tau^{-1}(w)|$.

  Let $\alpha_1,\ldots,\alpha_k$ be the colours in $\Phi(w)$ ordered so that
  \[
    N_{\alpha_1}(p,w) \ge N_{\alpha_2}(p,w) \ge\cdots\ge N_{\alpha_k}(p,w) \enspace .
  \]
  Then, for each $i\in\{1,\ldots,k\}$, $iN_{\alpha_i}(p,w)\le\sum_{j=1}^i N_{\alpha_j}(p,w) \le |\tau^{-1}(w)|$, so
  \[
    N_{\alpha_i}(p,w)\le \frac{|\tau^{-1}(w)|}{i}=\frac{O(\Delta^{\lfloor\ell/2\rfloor})}{i} \enspace .
  \]
  Therefore, regardless of any random choices made before choosing the colour of $w$ and any random choices made after choosing the colour of $w$, the random variable $X'_{p,w}$ is dominated by a random variable $X_{p,w}:=\min\{O(\Delta^{\lfloor\ell/2\rfloor})/i,M\}$ where $i$ is chosen uniformly in $\{1,\ldots,k\}$.

  Therefore, $|N_H^+(p)\cap P|$ is dominated by a sum $X_p:=\sum_{w} X_{p,w}$ of mutually independent random variables.  In order to apply a concentration result to the random variable $X_p$, we need to  establish sufficiently strong properties on the individual terms $X_{p,w}$.  In the next section, we bound the expectation and variance of each $X_{p,w}$ so that we can apply Bernstein's Inequality to prove that $\Pr(X_p\ge ck\log n)\le n^{-\Omega(c)}$.  Then the union bound implies that $\Pr(\max\{X_p:p\in V(H)\} \ge ck\log n)\le n^{-\Omega(c)}$.  Thus, with high probability, the number of additional colours in $\Phi_{q+1}$ needed to recolour the vertices of $P$ in the second phase is $O(k\log n)$.  Since the total number of colours used in the first phase is $O(k\log n)$, we have
  \[
    \lrn(G) = O(k \log n) = O(\Delta^{\lfloor\ell/2\rfloor-1/2}\log^{5/4} n) \enspace . \qedhere
  \]
\end{proof}

With the proof of \cref{l_degenerate_and_degree} out of the way, we now show how it implies \cref{l_d_degenerate_upper_bound}.

\begin{proof}[Proof of \cref{l_d_degenerate_upper_bound}]
  Since $G$ is $d$-degenerate, it has at most $dn$ edges and the sum of its vertex degrees is at most $2dn$.  Let $\Delta:=n^{1/(\lfloor\ell/2\rfloor+1/2)}\log^{-x} n$ for some value $x$ to be discussed shortly.  Then the set $S:=\{v\in V(G):\deg_G(v)\ge \Delta\}$ has size at most $2dn/\Delta=2dn^{\texp}\log^x n$.  Since $G-S$ is $d$-degenerate and has maximum degree $\Delta$, \cref{l_degenerate_and_degree} implies that
  \begin{align*}
    \lrn(G-S) &
    \le c\Delta^{\lfloor\ell/2\rfloor-1/2}\log^{5/4} n \\
    & = cn^{\frac{\lfloor\ell/2\rfloor-1/2}{\lfloor\ell/2\rfloor+1/2}}\log^{5/4-x(\lfloor\ell/2\rfloor-1/2)} n \\
    & = cn^{\dexp}\log^{5/4-x(\lfloor\ell/2\rfloor-1/2)} n \\
    & \le cn^{\dexp}\log^{5/4-x/2} n \\
  \end{align*}
  where the last inequality follows from the fact that $\lfloor\ell/2\rfloor-1/2\ge 1/2$ for all $\ell\ge 2$.  Taking $x:=5/6$, we get $|S|=O(n^{\texp}\log^{5/6} n)=O(n^{\texp}\log n)$ and  $\lrn(G-S)= O(n^{\texp}\log^{5/4-5/12} n)=O(n^{\texp}\log^{5/6} n)$.  We can extend $\varphi$ to a colouring of $G$ by assigning each vertex in $S$ a distinct colour that is larger than every colour used in the colouring of $G-S$.  Thus, $\trn(G)\le |S|+\lrn(G-S) \le O(n^{\texp}\log^{5/6} n)$, which establishes \cref{l_d_degenerate_upper_bound}.  (Note that this argument actually proves a slightly better bound than \cref{l_d_degenerate_upper_bound}.  For $\ell\ge 4$, further improvements to the logarithmic factor in \cref{l_d_degenerate_upper_bound} can be obtained using the ``Furthermore'' clause of \cref{l_degenerate_and_degree}.)
\end{proof}

\section{Bounding the Tail of \boldmath $X_p$}
\label{probability}

We make use of the following inequality of Bernstein \cite[Corollary~2.11]{boucheron.lugosi.ea:concentration}:

\begin{thm}\label{bernstein_theorem}
  Let $M$ be a positive number, let $X_1,\ldots,X_r$ be independent random variables such that $0\le X_i\le M$ for each $i\in\{1,\ldots,r\}$, and let $X:=\sum_{i=1}^r X_i$. Then
  \begin{equation}
    \Pr\left(X \ge \E(X)+ t\right)
      \le \exp\left(\frac{\tfrac{1}{2}t^2}{\sum_{i=1}^r \E((X_i-\E(X_i))^2)+\tfrac{1}{3}Mt}\right) \enspace . \label{bernstein}
  \end{equation}
\end{thm}
We will apply \cref{bernstein_theorem} to a random variable $X:=\sum_{i=1}^r X_i$ in which each $X_i$ has the following distribution (for some $0\le x_i\le kM$):
\[
  X_i = \begin{cases}
          M & \text{with probability $(1/k)\lfloor x_i/M\rfloor$} \\
          x_i/j & \text{with probability $1/k$ for each $j\in\{\lfloor x_i/M\rfloor+1,\ldots,k\}$}
        \end{cases}
\]
This is the distribution we get when we choose a uniform $j\in\{1,\ldots,k\}$ and set $X_i:=\min\{M,x_i/j)$.  To see how this applies in the proof of \cref{l_degenerate_and_degree}, let $r:=n$, let $\{w_1,\ldots,w_n\}:=V(G)$ and, for each $i\in\{1,\ldots,n\}$, let $x_i:=|\tau^{-1}(w_i)\cap\left(\bigcup_{y\in N^+_H(p)}\gamma^{-1}(y)\right)|$.  In our setting $k=\Delta^{\lfloor\ell/2\rfloor-1/2}\log^b n$ for some $b\ge 0$,  $M=a\Delta^{\lfloor\ell/2\rfloor}/k$ for some constant $a$, and $t=ck\log n$ for some (sufficiently large) constant $c$.  The rest of this appendix is devoted to bounding the various quantities that appear in \cref{bernstein} so that we can show that the right-hand side of \cref{bernstein} is $n^{-\Omega(c)}$.

Both the maximum value and the sum of $x_1,\ldots,x_n$ are important for us. For the maximum, we have $x_i\le|\tau^{-1}(w_i)|= O(\Delta^{\lfloor\ell/2\rfloor})$ for all $i\in\{1,\ldots,n\}$.  We have already bounded the sum when computing the expectation of $|N_H^+(p)\cap P|$ as follows:
\[
  \sum_{i=1}^n x_i = \sum_{y\in N^+_H(p)} |\gamma^{-1}(y)|
  \le |N^+_H(p)|\cdot O(\Delta^{\lfloor\ell/2\rfloor-1})
  \le O(\Delta^{2\lfloor\ell/2\rfloor-1}) \enspace .
\]

For each $i\in\{1,\ldots,n\}$, we have
\begin{align*}
  \E(X_i)
  & =\Pr(X_i=M)\cdot M + \sum_{j=\lfloor x_i/M\rfloor+1}^k \Pr(X_i=x_i/j)\cdot\frac{x_i}{j} \\
  & \le\frac{x_i}{kM}\cdot M + \frac{1}{k}\cdot\sum_{j=\lfloor x_i/M\rfloor+1}^k \frac{x_i}{j} \\
  & \le\frac{x_i}{k} + \frac{1}{k}\cdot\sum_{j=1}^k \frac{x_i}{j} \\
  & \le \frac{x_i(2+\ln k)}{k} & \text{(since $\sum_{j=1}^k 1/k\le 1+\ln k$)} \\
  & = O\left(\frac{x_i\log k}{k}\right) \\
  & = O\left(\frac{x_i\log n}{k}\right)
  \enspace ,
\end{align*}
where the last line comes from the fact that $d,\Delta \le n$ and $\ell\in O(1)$.
Therefore,
\begin{align}
  (\E(X_i))^2
  & = O\left(\left(\frac{x_i\log n}{k}\right)^2\right) \nonumber \\
  & = O\left(\frac{x_i^2\log^2 n}{k^2}\right) \nonumber \\
  & = O\left(\frac{x_i^2\log^2 n}{\Delta^{2\lfloor\ell/2\rfloor-1}\log^{2b} n}\right) \nonumber \\
  & = O\left(\frac{x_i\log^{2-2b} n}{\Delta^{\lfloor\ell/2\rfloor-1}}\right)
   & \text{(since $x_i=O(\Delta^{\lfloor\ell/2\rfloor})$)} \label{square}
\end{align}
Now,
\begin{equation}
  \E((X_i-\E(X_i))^2)  = \frac{1}{k}\left\lfloor\frac{x_i}{M}\right\rfloor\cdot(M-\E(X_i))^2 + \sum_{j=\lfloor x_i/M\rfloor+1}^k \frac{1}{k}\left(\frac{x_i}{j}-\E(X_i)\right)^2 \label{variance}
\end{equation}
We bound the first term on the right hand side of \cref{variance} as follows:
\begin{align*}
  \frac{1}{k}\left\lfloor\frac{x_i}{M}\right\rfloor\cdot(M-\E(X_i))^2
  & \le \frac{x_i}{kM}\left(M^2 + (\E(X_i))^2\right) \\
  & = \frac{M x_i}{k} + \frac{x_i}{kM}\cdot (\E(X_i))^2 \\
  & \le \frac{M x_i}{k} + (\E(X_i))^2
  & \text{(since $x_i=|\tau^{-1}(w_i)|\le Mk$)}\\
  & = O\left(\frac{\Delta^{\lfloor\ell/2\rfloor}x_i}{k^2}\right) + (\E(X_i))^2
  & \text{(by the definition of $M$)}\\
  & = O\left(\frac{x_i\log^{-2b} n}{\Delta^{\lfloor\ell/2\rfloor-1}}\right) + (\E(X_i))^2
  & \text{(by the definition of $k$)}\\
  & = O\left(\frac{x_i\log^{2-2b} n}{\Delta^{\lfloor\ell/2\rfloor-1}}\right)
  & \text{(by \cref{square}, since $2-2b\ge -2b$)} \enspace .
\end{align*}
We bound the remaining terms in \cref{variance} as follows:
\begin{align*}
 \sum_{j=\lfloor x_i/M\rfloor+1}^k \frac{1}{k}\left(\frac{x_i}{j}-\E(X_i)\right)^2
  & \le \sum_{j=\lfloor x_i/M\rfloor+1}^k \frac{1}{k}\left(\frac{x_i^2}{j^2}+(\E(X_i))^2\right) \\
  & \le (\E(X_i))^2+\frac{1}{k}\sum_{j=\lfloor x_i/M\rfloor+1}^k \frac{x_i^2}{j^2} \\
  & \le (\E(X_i))^2 +  \frac{1}{k}\cdot\frac{M\pi^2x_i^2}{6 x_i}
  & \text{(since $\sum_{j=z}^{\infty} \tfrac{1}{j^2} \le \tfrac{\pi^2}{6z}$ for $z\ge 1$)} \\
  & = (\E(X_i))^2 + O\left(\frac{Mx_i}{k}\right) \\
  & = O\left(\frac{x_i\log^{2-2b} n}{\Delta^{\lfloor\ell/2\rfloor-1}}\right)
  & \text{(as above)} \enspace .
\end{align*}
Therefore
\[
  \E((X_i-\E(X_i))^2) = O\left(\frac{x_i\log^{2-2b} n}{\Delta^{\lfloor\ell/2\rfloor-1}}\right)
\]
for all $i\in\{1,\ldots,n\}$.  Recall that $\sum_{i=1}^n x_i = O(\Delta^{2\lfloor\ell/2\rfloor-1})$, so
\[
  \E(X)
  = \sum_{i=1}^n \E(X_i)
  = \sum_{i=1}^n O\left(\frac{x_i\log n}{k}\right)
  = O\left(\frac{\Delta^{2\lfloor\ell/2\rfloor-1}\log n}{k}\right)
  = O\left(\Delta^{\lfloor\ell/2\rfloor-1/2}\log^{1-b} n\right) \enspace .
\]
and
\[
  \sum_{i=1}^r\E((X_i-\E(X_i))^2)
  = \sum_{i=1}^r O\left(\frac{x_i\log^{2-2b} n}{\Delta^{\lfloor\ell/2\rfloor-1}}\right)
  = O\left(\frac{\Delta^{2\lfloor\ell/2\rfloor-1}\log^{2-2b} n}{\Delta^{\lfloor\ell/2\rfloor-1}}\right)
  = O\left(\Delta^{\lfloor\ell/2\rfloor}\log^{2-2b} n\right)
  \enspace .
\]
Then \cref{bernstein} gives:
\begin{align*}
  \Pr(X\ge \E(X)+t)
  & = \Pr(X\ge \E(X)+ck\log n) \\
  & \le \exp \left(-\frac{\tfrac{1}{2}(ck\log n)^2}{O\left(\Delta^{\lfloor\ell/2\rfloor}\log^{2-2b} n\right) + \tfrac{1}{3}Mck\log n}\right) \\
  & = \exp \left(-\frac{\tfrac{1}{2}(ck\log n)^2}{O\left(\Delta^{\lfloor\ell/2\rfloor}\log^{2-2b} n +c\Delta^{\lfloor\ell/2\rfloor}\log n)\right)}\right)
    & \text{(since $Mk=O(\Delta^{\lfloor\ell/2\rfloor})$)} \\
  & = \exp \left(-\frac{\tfrac{1}{2}(ck\log n)^2}{O\left(c\Delta^{\lfloor\ell/2\rfloor}\log^{2-2b} n\right)}\right) & (\text{for $c\ge 1$ and $b\le 1/2$}) \\
  & = \exp \left(-\frac{\tfrac{1}{2}c\Delta^{2\lfloor\ell/2\rfloor-1}\log^{2+2b} n}{O\left(\Delta^{\lfloor\ell/2\rfloor}\log^{2-2b} n\right)}\right) & \text{(by definition of $k$)}\\
  & = \exp\left(-\Omega(c\Delta^{\lfloor\ell/2\rfloor-1}\log^{4b} n)\right) \\
  & = n^{-\Omega(c)} \enspace ,
\end{align*}
provided that $b\ge 1/4$ or $\Delta^{\lfloor\ell/2\rfloor-1}\ge\log n$.
For any fixed $c$, taking $b=1/4$ gives
\begin{align*}
  \E(X)+t
    & = O(\Delta^{\lfloor\ell/2\rfloor-1/2}\log^{1-b} n + k\log n) \\
    & = O(\Delta^{\lfloor\ell/2\rfloor-1/2}(\log^{1-b} n + \log^{1+b} n)) \\
    & = O(\Delta^{\lfloor\ell/2\rfloor-1/2}\log^{5/4} n)
\end{align*}
Thus, $\Pr(X\ge c\Delta^{\lfloor\ell/2\rfloor-1/2}\log^{5/4} n) \le n^{-\Omega(c)}$, which completes the first part of the proof of \cref{l_degenerate_and_degree} for the cases $\ell=2$ and $\ell=3$.  To complete the ``furthermore'' clause of the proof we take $b=0$ and deduce that $\E(X)+t=O(\Delta^{\lfloor\ell/2\rfloor-1/2}\log n)$.

\section*{Acknowledgement}

Some of this research began during the \emph{Third Workshop on Geometry and Graphs (WoGaG~2015)}, March 8–13, 2015, and resumed during the \emph{Second Adriatic Workshop on Graphs and Probability (AWGP~2023)}, Jun 24–Jul 1, 2023. The authors are grateful to the organizers of both workshops for providing a stimulating and productive work environment.

\bibliographystyle{plainnat}
\bibliography{us2}

\begin{thebibliography}{9}
\providecommand{\natexlab}[1]{#1}
\providecommand{\url}[1]{\texttt{#1}}
\expandafter\ifx\csname urlstyle\endcsname\relax
  \providecommand{\doi}[1]{doi: #1}\else
  \providecommand{\doi}{doi: \begingroup \urlstyle{rm}\Url}\fi

\bibitem[Almeter et~al.(2019)Almeter, Demircan, Kallmeyer, Milans, and
  Winslow]{almeter.demircan.ea:graph}
Jordan Almeter, Samet Demircan, Andrew Kallmeyer, Kevin~G. Milans, and Robert
  Winslow.
\newblock Graph 2-rankings.
\newblock \emph{Graphs and Combinatorics}, 35\penalty0 (1):\penalty0 91--102,
  2019.
\newblock \doi{10.1007/s00373-018-1979-4}.
\newblock URL \url{https://doi.org/10.1007/s00373-018-1979-4}.

\bibitem[Bose et~al.(2020)Bose, Dujmovi{\'c}, Javarsineh, and
  Morin]{bose.dujmovic.ea:asymptotically}
Prosenjit Bose, Vida Dujmovi{\'c}, Mehrnoosh Javarsineh, and Pat Morin.
\newblock Asymptotically optimal vertex ranking of planar graphs.
\newblock \emph{CoRR}, abs/2007.06455, 2020.
\newblock URL \url{https://arxiv.org/abs/2007.06455}.

\bibitem[Boucheron et~al.(2013)Boucheron, Lugosi, and
  Massart]{boucheron.lugosi.ea:concentration}
Stéphane Boucheron, Gábor Lugosi, and Pascal Massart.
\newblock \emph{Concentration Inequalities: A Nonasymptotic Theory of
  Independence}.
\newblock Oxford University Press, Oxford, United Kingdom, first edition, 2013.

\bibitem[Devroye et~al.(2019)Devroye, Dujmovic, Frieze, Mehrabian, Morin, and
  Reed]{devroye.dujmovic.ea:notes}
Luc Devroye, Vida Dujmovic, Alan~M. Frieze, Abbas Mehrabian, Pat Morin, and
  Bruce~A. Reed.
\newblock Notes on growing a tree in a graph.
\newblock \emph{Random Struct. Algorithms}, 55\penalty0 (2):\penalty0 290--312,
  2019.
\newblock \doi{10.1002/rsa.20828}.
\newblock URL \url{https://doi.org/10.1002/rsa.20828}.

\bibitem[Diestel(2012)]{diestel:graph}
Reinhard Diestel.
\newblock \emph{Graph Theory, 4th Edition}, volume 173 of \emph{Graduate texts
  in mathematics}.
\newblock Springer, 2012.
\newblock ISBN 978-3-642-14278-9.

\bibitem[Karpas et~al.(2015)Karpas, Neiman, and
  Smorodinsky]{karpas.neiman.ea:on}
Ilan Karpas, Ofer Neiman, and Shakhar Smorodinsky.
\newblock On vertex rankings of graphs and its relatives.
\newblock \emph{Discret. Math.}, 338\penalty0 (8):\penalty0 1460--1467, 2015.
\newblock \doi{10.1016/j.disc.2015.03.008}.
\newblock URL \url{https://doi.org/10.1016/j.disc.2015.03.008}.

\bibitem[Nešetřil and Ossona~de Mendez(2006)]{nesetril.ossona:tree-depth}
Jaroslav Nešetřil and Patrice Ossona~de Mendez.
\newblock Tree-depth, subgraph coloring and homomorphism bounds.
\newblock \emph{Eur. J. Comb.}, 27\penalty0 (6):\penalty0 1022--1041, 2006.
\newblock \doi{10.1016/j.ejc.2005.01.010}.
\newblock URL \url{https://doi.org/10.1016/j.ejc.2005.01.010}.

\bibitem[Nešetřil and Ossona~de Mendez(2012)]{nesetril.ossona:sparsity}
Jaroslav Nešetřil and Patrice Ossona~de Mendez.
\newblock \emph{Sparsity: Graphs, Structures, and Algorithms}, volume~28 of
  \emph{Algorithms and combinatorics}.
\newblock Springer, 2012.
\newblock ISBN 978-3-642-27874-7.
\newblock \doi{10.1007/978-3-642-27875-4}.
\newblock URL \url{https://doi.org/10.1007/978-3-642-27875-4}.

\bibitem[Shalu and Antony(2020)]{shalu.antony:complexity}
M.~A. Shalu and Cyriac Antony.
\newblock Complexity of restricted variant of star colouring.
\newblock In Manoj Changat and Sandip Das, editors, \emph{Algorithms and
  Discrete Applied Mathematics - 6th International Conference, {CALDAM} 2020,
  Hyderabad, India, February 13-15, 2020, Proceedings}, volume 12016 of
  \emph{Lecture Notes in Computer Science}, pages 3--14. Springer, 2020.
\newblock \doi{10.1007/978-3-030-39219-2\_1}.
\newblock URL \url{https://doi.org/10.1007/978-3-030-39219-2\_1}.

\end{thebibliography}

\end{document}